\documentclass[11pt]{amsart}

\usepackage{rotating}
\usepackage{pinlabel}

\newtheorem{theorem}{Theorem}[section]

\newtheorem{lemma}[theorem]{Lemma}

\newtheorem{cor}[theorem]{Corollary}
\newtheorem{remark}[theorem]{Remark}
\newtheorem{example}[theorem]{Example}

\newcommand{\cee}{\mathbb{C}}

\newcommand{\W}{{\mathcal W}}

\newcommand{\N}{{\mathcal N}}
\newcommand{\C}{{\mathcal C}}
\newcommand{\F}{{\mathcal F}}

\newcommand{\B}{\mbox{\bf b}}
\newcommand{\A}{\mbox{\bf a}}
\newcommand{\ww}{\mbox{\bf w}}

\usepackage{graphicx,amsfonts}

\begin{document}

\title{Convex plumbings and Lefschetz fibrations}

\author{David Gay}
\author{Thomas E. Mark}
\thanks{This work was partially supported by a grant from the Simons Foundation (\#210381 to David Gay). The second author was partially supported by NSF grant DMS-0905380.}
\begin{abstract}
We show that under appropriate hypotheses, a plumbing of symplectic surfaces in a symplectic 4-manifold admits strongly convex neighborhoods. Moreover the neighborhoods are Lefschetz fibered with an easily-described open book on the boundary supporting the induced contact structure. We point out some applications to cut-and-paste constructions of symplectic 4-manifolds.
\end{abstract}

\maketitle

\section{Introduction}

Strong symplectic convexity is an essential feature of most $4$--dimensional symplectic cut-and-paste operations, which have been particularly important in recent constructions of ``small'' exotic $4$--manifolds (\cite{baldkirk}, \cite{SSP}, \cite{jongil}). For such an operation to work, one needs to:
\begin{enumerate}
 \item {\em recognize} a symplectic $4$--manifold $(Z,\eta)$ with strongly convex boundary inside an ambient symplectic $4$--manifold $(X,\omega)$ and characterize the induced contact structure $\xi$ on $\partial Z$, 
 \item {\em construct} a replacement symplectic $4$-manifold $(Z',\eta')$ with strongly convex boundary and characterize the induced contact structure $\xi'$ on $\partial Z'$, and finally 
 \item recognize that $(\partial Z,\xi)$ is {\em contactomorphic} to $(\partial Z', \xi')$.
\end{enumerate}
Then $(X\setminus Z) \cup Z'$ inherits a symplectic structure from $\omega$ and $\eta'$. 

The main result of this paper is a general method for implementing step~1 that dovetails with a standard method for implementing step~2 (constructing
Stein surfaces as Lefschetz fibrations over $D^2$) and a standard method for
characterizing contact $3$-manifolds (via open book decompositions) and thus
for carrying out step~3.

Let $\C = C_1 \cup \ldots \cup C_n$ be a configuration of symplectic surfaces in $(X,\omega)$ intersecting $\omega$--orthogonally according to a connected plumbing graph $\Gamma$, with {\em negative definite} intersection form $Q = (q_{ij}) = ([C_i] \cdot [C_j])$. Suppose that for each row in $Q$, we have a nonpositive row sum $s_i = \sum_j q_{ij} \leq 0$. For convenience, we also assume that $\Gamma$ contains no edges connecting a vertex with itself. Let
$\Sigma$ be the result of connect-summing $|s_i|$ copies of $D^2$ to each $C_i$ and then
connect-summing these surfaces according to $\Gamma$. Let $\{c_1, \ldots, c_k\}$ be the collection of simple closed curves on $\Sigma$ consisting of one curve around each connect-sum neck, with $\tau$ equal to the product of right Dehn twists along these curves (see Figure~\ref{GraphPageCurves}). 

\begin{theorem} \label{T:Main}
Any neighborhood of $\C$ contains a neighborhood $(Z,\eta)$ of $\C$ with  strongly convex boundary, that admits a symplectic Lefschetz fibration $\pi : Z \to D^2$ having regular fiber $\Sigma$ and exactly one singular fiber $\Sigma_0 = \pi^{-1}(0)$. The vanishing cycles are $c_1, \ldots, c_k$ and $\C$ is the union of the closed components of $\Sigma_0$. The induced contact structure $\xi$ on $\partial Z$ is supported by the induced open book $(\Sigma,\tau)$.
\end{theorem}

\begin{figure}
\labellist
\small\hair 2pt
\pinlabel $(0,-2)$ [r] at 37 56
\pinlabel $(1,-2)$ [r] at 36 4
\pinlabel $(1,-4)$ [l] at 105 31
\pinlabel $(0,-5)$ [bl] at 188 35
\pinlabel $(2,-2)$ [l] at 250 33
\endlabellist
\centering

 \includegraphics{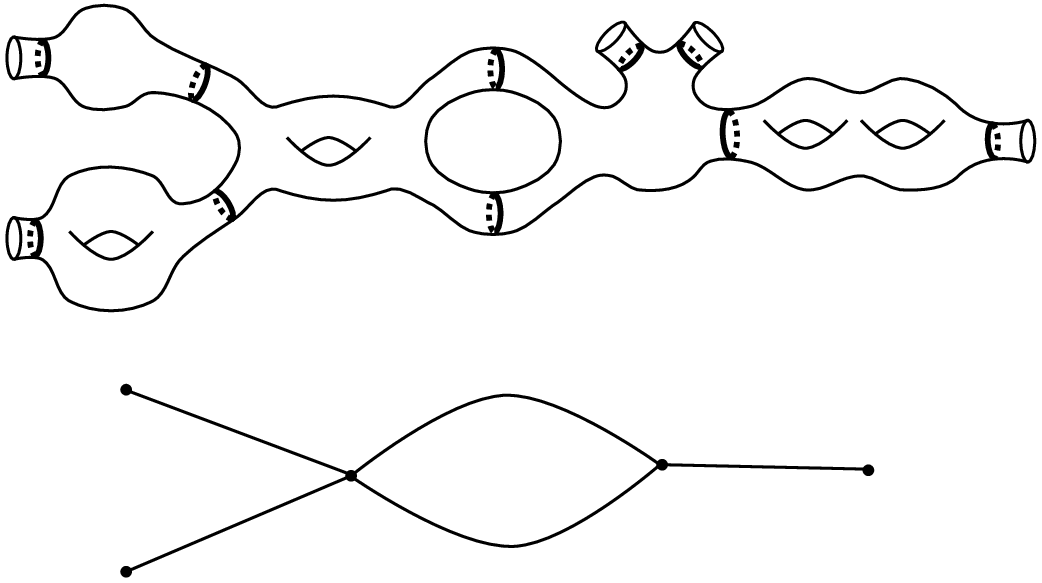}
 \caption{\label{GraphPageCurves} A plumbing graph satisfying the hypotheses of Theorem~\ref{T:Main}, with the associated surface $\Sigma$ and vanishing cycles $c_1, \ldots, c_{10}$. Vertices of the graph are labeled by (genus, self-intersection).}
\end{figure}
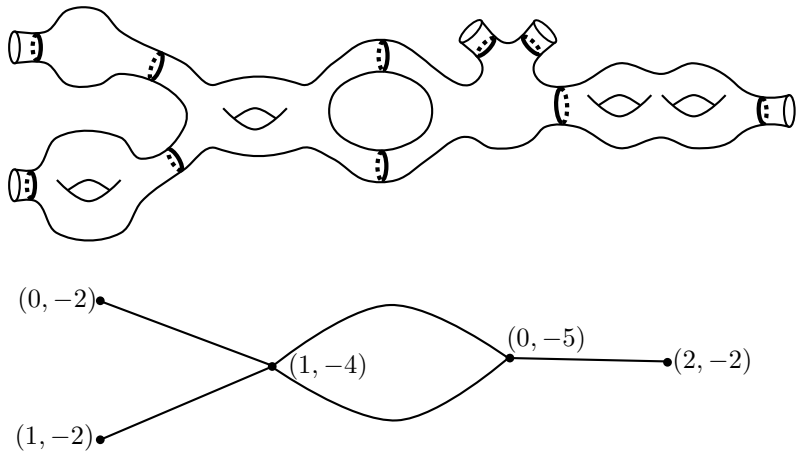

We now comment on the importance of this result and how it fits in with known results to give a very general package for symplectic cut-and-paste operations in dimension $4$. 

First note that our theorem {\em recognizes} the codimension-$0$ submanifold $(Z, \eta)$ by recognizing a collection of surfaces intersecting according to a plumbing graph, and it is obviously easier to identify a codimension-$2$ object (the configuration of surfaces) than to directly identify a codimension-$0$ object. Thus this is an effective method for implementing step~1 of the procedure outlined above and such a method, starting from configurations of surfaces, has been at the core of most $4$--dimensional symplectic cut-and-paste operations to date.

The strong convexity part of our main result was proved in \cite{gaystip}, but, without the Lefschetz fibration and open book, such a result seemed to be useful for cut-and-paste operations only when one had strong classification results for tight contact structures on the $3$--manifold $\partial Z$.

It is clear that being able to characterize $(\partial Z, \xi)$ using an open book is valuable as long as a good method exists to construct the replacement $(Z',\omega')$ which yields an open book characterization of $(\partial Z', \xi')$. Such a method does exist, using Eliashberg-Weinstein handle attachments along Legendrian knots to construct compact Stein surfaces (\cite{akboz}, \cite{eliash2}, \cite{gompf}, \cite{loipier}, \cite{weinstein}). 

In particular, the monodromy substitution techniques of Endo--Mark--Van Horn-Morris \cite{EMV} become symplectic operations thanks to this result. More generally, we have the following corollary:

\begin{cor}
 In the setting of Theorem~\ref{T:Main}, suppose that $c'_1, \ldots, c'_{k'}$ is a sequence of homologically essential simple closed curves in $\Sigma$ such that the product $\tau'$ of Dehn twists along these curves is isotopic to $\tau$. Let $Z'$ be the $4$--manifold with smooth Lefschetz fibration over $D^2$ having regular fiber $\Sigma$ and vanishing cycles $c'_1, \ldots, c'_{k'}$ (on disjoint fibers). Then $(X \setminus Z) \cup Z'$ supports a symplectic form $\eta$ inherited from $\omega$ on $X$ and $\omega'$ on $Z'$.
\end{cor}

\begin{proof}
In this case it follows from \cite{akboz}, \cite{eliash2}, \cite{honda}, \cite{weinstein}  that $Z'$ supports a symplectic form $\omega'$ with strongly convex boundary (in fact, $Z'$ admits a Stein structure), and the induced contact structure $\xi'$ on $\partial Z'$ is supported by the induced open book. Since $\tau = \tau'$, this is the same as the open book on $\partial Z$, and hence $(\partial Z',\xi')$ is contactomorphic to $(\partial Z, \xi)$. Since $X \setminus Z$ is concave  and $Z'$ is convex, the gluing to construct $(X \setminus Z) \cup Z'$ can be performed symplectically by a standard argument.
\end{proof}

In other words, if $\tau$ is a composition of disjoint twists as in the theorem, the relation $\tau = \tau'$ in the mapping class group of $\Sigma$ immediately gives rise to a symplectic cut-and-paste operation whereby the Lefschetz fibration described by $\tau$---namely a negative-definite plumbing---is replaced by that described by $\tau'$. Here and below we are making use of the fact that a Lefschetz fibration over $D^2$ is uniquely specified by a sequence of curves on the generic fiber, which describe the monodromy of the fibration as a sequence of Dehn twists.

As a first example, consider a torus with one boundary component $d$, and let $a$ and $b$ be standard generators of the first homology. We have a well-known relation $t_d = (t_at_b)^6$, which for the sake of interest we square to $t_d^2 = (t_at_b)^{12}$. The multitwist $\tau = t_d^2$ describes a neighborhood of the configuration $\C$ consisting of a torus of square $-1$ plumbed with a sphere of square $-2$, while the word $\tau' = (t_at_b)^{12}$ describes a Stein manifold diffeomorphic to the complement of a fiber $F$ and section $S$ in an elliptic surface $E(2)$. Hence we have a symplectic surgery according to the corollary, whereby a symplectic manifold $X$ containing a copy of $\C$ gives rise to a new symplectic manifold $\widetilde{X} = (X - nbd(\C)) \cup_\partial (E(2) - nbd(F\cup S))$. Indeed, one might think of this as a fiber sum of $Z$ with $E(2)$ along the singular surface $\C\subset X$ and the ``dual configuration'' $F\cup S\subset E(2)$. 

Similarly, the relations of \cite{KO} give rise (after squaring) to analogous fiber sum operations along star-shaped plumbing graphs: for $k\in\{1,\ldots,9\}$, Korkmaz and Ozbagci \cite{KO} give us explicit relations $\tau_k = \tau'_k$ where $\tau_k$ is a multitwist on a torus with $k$ boundary components describing a plumbing as in Figure \ref{stargraph}, and where $\tau'_k$ describes the complement of $\C'=F\cup S_1\cup \cdots \cup S_k$ in $E(2)$ for disjoint sections $S_1,\ldots,S_k$. In other words, if $(X,\omega)$ contains a symplectic copy of the configuration $\C$ described in Figure \ref{stargraph}, then the operation $X\leadsto X' = (X\setminus \C)\cup_\partial (E(2) \setminus \C')$ is symplectic. 
\begin{figure}
 \includegraphics[width=2.3in]{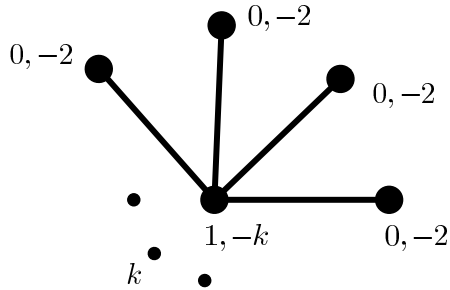}
\caption{\label{stargraph} Plumbing graph. Vertices are labeled by (genus, self-intersection); there are $k$ legs, $k\in\{1,\ldots,9\}$.}
\end{figure}

In another direction, Endo--Mark--Van Horn-Morris \cite{EMV} describe several families of relations of the required form, in which $\tau$ is a multitwist on a planar surface, and $\tau'$ gives rise to a rational ball. We infer that the plumbings described by these words $\tau$, comprising the linear plumbings of Fintushel-Stern and Park as well as the three-legged graphs in the families $\W$ and $\N$ of Stipsicz, Szab\'o, and Wahl \cite{SSW}, can be replaced symplectically by the rational ball described by the corresponding $\tau'$. In other words, this gives a new proof that these rational blowdowns may be performed symplectically.

\begin{proof}[Proof of Theorem \ref{T:Main}] Given a plumbing graph $\Gamma$, we write $X_\Gamma$ for the associated plumbed 4-manifold. In Theorems \ref{convthm} and \ref{thm2} below we construct a ``model'' symplectic structure $\omega$ on $X_\Gamma$, along with a Lefschetz fibration $\pi: X_\Gamma\to D^2$, satisfying several requirements. First, the configuration $\C_\Gamma = \bigcup_{v\in\Gamma} C_v$ is symplectically embedded (with orthogonal intersections), with arbitrary prescribed areas for the individual curves $C_v$. The Lefschetz fibration on $X_\Gamma$ has symplectic fibers, and moreover each fibered neighborhood $\pi^{-1}(D^2(\epsilon))$ is convex. In particular any neighborhood of $\C_\Gamma$ contains a convex neighborhood. Finally, we check that the contact structures on the boundary of fibered neighborhoods induced from the convexity and from the Lefschetz fibration agree. From this the theorem follows by an application of a neighborhood theorem for configurations of symplectic curves (see, e.g., \cite{mcrae}).\end{proof}

Recall that a symplectic manifold is convex if it admits a Liouville field pointing transversely out of the boundary, which entails in particular that the symplectic form be exact on the subset where the Liouville field is defined. If $D\to C_v$ is a disk bundle, we wish to find a symplectic form $\omega$ on $D$ for which the zero-section $C_v\subset D$ is symplectic, but has arbitrarily small convex neighborhoods. Thus we look for a Liouville field defined on $D-C_v$, which must have a singularity along $C_v$ itself since $\omega$ is not exact on $D$. Recall that constructing a plumbing of two such disk bundles $D',D''$ entails exchanging the factors in local trivializations of $D'$ and $D''$. We wish to find a new Liouville field over the plumbed manifold; hence we should arrange that the Liouville field on $D'$ also has a singularity along the normal disk over the plumbing point, and correspondingly for $D''$. 

In Lemma \ref{basiclemma} below we show how to construct Liouville fields with the necessary singularities on disk bundles, with local forms adapted to plumbings. This technique of {\it convex plumbing} quickly establishes the existence of an appropriate symplectic form on $X_\Gamma$, as laid out in Theorem \ref{convthm}. The subsequent section revisits the construction and imposes the additional structure of a Lefschetz fibration to give Theorem \ref{thm2}.

\subsection*{Additional remarks} The idea motivating this work was to use the results of \cite{EMV} realizing (smooth) rational blowdowns via monodromy substitutions to give a new proof that these are symplectic operations: in particular, to use the Lefschetz fibration on the neighborhood of a plumbing. One looks for the necessary convexity, and this seems at first sight, perhaps, to be straightforward: for example, the handle description provided by the Lefschetz fibration shows that the neighborhood of a plumbing admits a Stein structure, and hence is convex. However, the Stein structure is not the one that models a plumbing of {\it symplectic} surfaces, since its symplectic form is necessarily exact. Alternatively, one could imagine using Gompf's results on the existence of symplectic forms on Lefschetz fibrations to provide the desired model: this gives the correct symplectic neighborhood of a plumbing, but now the convexity is no longer clear. The point of the constructions in this paper is to arrange for the various requirements to hold simultaneously: a symplectic form giving the prescribed structure on the plumbed curves so as to use the neighborhood theorem, a Lefschetz fibration with symplectic fibers to control the contact structure on the boundary, and an appropriate Liouville field to give convexity.

\section{Convex Plumbing}

Consider a smooth closed oriented 2-manifold $C$ of genus $g$, and let $m$ and $n_1,\ldots, n_m$ be positive integers. Choose a positive real number $A$, and fix $m$ additional nonnegative constants $A_1,\ldots,A_m$ with the property that $A > A_1/n_1 + \cdots + A_m/n_m$. Let $p:D\to C$ be the oriented disk bundle of degree $-m$ over $C$, and also fix $m$ marked points $x_1,\ldots,x_m\in C$. 

\begin{lemma}\label{basiclemma} There exists a symplectic form $\omega$ on $D$ with the following properties.
\begin{itemize}
\item The zero section $C\subset D$ is a symplectic submanifold with area 
\[\pi(A - A_1/n_1-\cdots - A_m/n_m),
\]
 and the fibers of $p$ are symplectic.
\item  Let $\F$ denote the union of those fibers $D_{x_i}$ of $D$ for which the corresponding $A_i > 0$. There is a neighborhood $U$ of $C$ and a Liouville vector field $V$ defined on $U - (C\cup \F)$.
\item There are local Darboux trivializations of $D$ over disjoint neighborhoods $W_i\subset C$ of each $x_i$, with respect to which $V$ appears as
\[
V = \frac{1}{2}(r_1 + A_i/n_ir_1) \partial_{r_1} + \frac{1}{2}(r_2 + A/mr_2)\partial_{r_2}.
\]
Here $(r_1,\theta_1)$ are coordinates on $W_i$ and $(r_2,\theta_2)$ are fiberwise coordinates on $D|_{W_i}$.
\item Every neighborhood of $C\cup \F$ contains a convex neighborhood, in particular a neighborhood for which $V$ is outward-pointing and transverse along the boundary.
\end{itemize}
\end{lemma}

\begin{proof} Choose a constant $\delta> 0$ such that $m\delta < \frac{1}{2}(A - A_1/n_1 - \cdots - A_m/n_m)$. Also fix an additional point $x_0 \in C$ distinct from the other $x_i$. Fix disjoint disk neighborhoods $W_0,W_1,\ldots, W_m$ of each marked point; then we claim we can find an area form $\beta$ on $C - x_0$ such that:
\begin{enumerate}
\item In suitable polar coordinates on $W_1,\ldots W_m$, we have that $W_i$ is the disk of radius $\sqrt{\delta}$ and $\beta = r_1dr_1\,d\theta_1$.
\item Under an identification of the deleted neighborhood $W_0 - x_0$ with the annulus $(\sqrt{A - m\delta}, \sqrt{A}) \times S^1$, $\beta$ also takes the form $\beta = r_1dr_1\,d\theta_1$.
\item The total area of $\beta$ is $\pi(A-A_1/n_1 - \cdots - A_m/n_m)$.
\end{enumerate}
Moreover, there is a Liouville field $X$ on $C - \{x_0,\ldots,x_m\}$ such that
\begin{enumerate}\setcounter{enumi}{3}
\item Near $x_0$, we have $X = \frac{1}{2}r_1\partial_{r_1}$.
\item Near $x_i$, $i>0$, we have $X = \frac{1}{2}(r_1 + A_i/n_ir_1)\partial_{r_1}$.
\end{enumerate}
To do this, first note that the desired area form has area $\pi\delta$ on each $W_i$, $i = 1,\ldots, m$, and area $m\pi\delta$ on $W_0 - x_0$. Hence from the condition on $\delta$, an area form on $C - x_0$ satisfying (1), (2) and (3) exists: define $\beta$ locally by the given conditions, and extend arbitrarily to an area form of the required area.  Note that in (2) our convention is that $\lim_{r_1\to \sqrt{A}}(r_1,\theta) = x_0$, so geometrically one can picture $W_0$ as an expanding cylinder attached to $C - nbd(x_0)$.

Now we must see that we can extend the given Liouville fields over the complement of ${\mathcal W} = W_0\cup \cdots \cup W_m$. Equivalently, we must find a primitive for $\beta$ extending the 1-form $\alpha = \iota_{X}\beta \in \Omega^1(\W)$. The existence of such a primitive follows from a Mayer-Vietoris argument together with a simple calculation in local coordinates showing that $\int_{\partial(C-\W)} \alpha = \int_{C-\W}\beta$.

Next we construct the bundle $D$ as a symplectic manifold. Let $Z_0 = (C- x_0)\times D^2(\sqrt{\delta})$ and $\omega_0 = \beta + r_2dr_2\,d\theta_2$, where $D^2(\sqrt{\delta})$ is the disk of radius $\sqrt{\delta}$ with coordinates $(r_2,\theta_2)$. On $Z_0$ we have a Liouville field $V_0 = X + \frac{1}{2}(r_2 + A/mr_2)\partial_{r_2}$, defined on the complement of $C$ and $\{x_1,\ldots,x_m\}\times D^2(\sqrt{\delta})$. 

Let $Z_1 = D^2(\sqrt{m\delta}) \times D^2(\sqrt{\delta})$, with standard symplectic form $\omega_1 = r_1dr_1\,d\theta_1 + r_2dr_2\,d\theta_2$ and Liouville field $V_1 = \frac{1}{2}(r_1\partial_{r_1} + (r_2 + A/mr_2)\partial_{r_2})$. 

Let $T\subset Z_1$ denote the subset $\{0< r_1^2 < m\delta,\, 0\leq r_2^2 < r_1^2/m\}$, and define an embedding $\phi: T\to W_0 \times D^2(\sqrt{\delta})\subset Z_0$ by
\[
\phi(r_1,\theta_1,r_2,\theta_2) = (\sqrt{-r_1^2 + mr_2^2 + A}, -\theta_1, r_2, m\theta_1+\theta_2).
\]
One checks that $\phi$ is a symplectic embedding of $T$ onto $W_0\times D^2(\sqrt{\delta})$, and intertwines the Liouville fields $V_0$, $V_1$. Moreover, $Z_0\cup_\phi Z_1$ is diffeomorphic to the disk bundle $D$ of degree $-m$ over $C$. Now, our gluing does not respect the obvious disk bundle projections $p_0, p_1$ on $Z_0$ and $Z_1$. However, it is not difficult to modify the map $p_0\circ \phi$, defined a priori on $T\subset Z_1$, into a disk bundle projection on $Z_1$ having symplectic fibers.

This proves all but the last claim in the statement; the latter we defer until the next section.
\end{proof}

 Let $\Gamma$ be a finite connected graph whose vertices $v$ are labeled with 
genera $g_v$ and weights $-m_v$. Suppose that $\Gamma$ satisfies
\begin{enumerate}
\item[i.] $\Gamma$ contains no loops, i.e., no edges from a vertex to itself.
\item[ii.] For each $v$ we have $m_v \geq d_v$, where $d_v$ is the degree (valence) of $v$.
\item[iii.] For any vector $\B$ of positive real numbers, there is another vector of positive numbers $\A$ such that $-Q_\Gamma\A = \B$, where $Q_\Gamma$ is the intersection (incidence) matrix of $\Gamma$.
\end{enumerate}
Note that given (ii), condition (iii) is equivalent to $Q_\Gamma$ being negative-definite: see Lemma \ref{deflemma} below.

Then we have the following convexity result.

\begin{theorem}\label{convthm} With $\Gamma$ as above, let $X_\Gamma$ be the plumbing of disk bundles having a disk bundle of degree $m_v$ over a surface $C_v$ of genus $g_v$ for each vertex $v$ of $\Gamma$, plumbed according to the edges of $\Gamma$. Then $X$ admits a symplectic form $\omega$ with the following properties.
\begin{itemize}
\item Each surface $C_v$ is symplectic, with arbitrary prescribed area $B_v$, and each intersection among the $C_v$ is $\omega$-orthogonal.
\item There is a Liouville field $V$ for $\omega$ defined on $X_\Gamma - \bigcup_v C_v$.
\item Any neighborhood of $\bigcup_v C_v$ contains a convex neighborhood; in particular a neighborhood for which $V$ is outward-pointing and transverse along the boundary. 
\end{itemize}
\end{theorem}

\begin{proof} Consider first the situation of constructing the symplectic plumbing of two disk bundles of degree $-m'$ and $-m''$ over two surfaces $C',C''$, with $m', m'' \geq 1$. Choose positive numbers $A',A''$; then two applications of the Lemma \ref{basiclemma} (with the roles of the primes exchanged, and also using the same value of $\delta$---see below) provides us with two symplectic disk bundles $D',D''$ over $C',C''$, and a marked point $x',x''$ on each. (If $m'>1$ or $m''>1$, we take the quantities $A_i$ associated to the remaining marked points to equal 0.) Moreover, the local trivializations of $D',D''$ near $x',x''$ provided by the lemma show that the plumbing diffeomorphism exchanging the two factors in the 
trivializations is a symplectomorphism intertwining the two Liouville fields $V',V''$ given by the lemma. Hence the lemma gives us the desired structure, so long as we can choose $A',A''$ such that 
\begin{eqnarray*}
\pi(A' - A''/m'') &=& B' \\
\pi(A'' - A'/m') &=& B'',
\end{eqnarray*}
where $B',B''$ are the specified symplectic areas of $C',C''$.

More generally, we wish to find constants $A_v > 0$ for each vertex $v$, with the following property. Fix a vertex $v_0$ and let $v_1,\ldots,v_d$ be the adjacent vertices (repeated, if more than one edge connects them with $v_0$). Then we must arrange 
\[
\pi(A_{v_0} - A_{v_1}/m_{v_1} - \cdots -A_{v_d}/m_{v_d}) = B_{v_0}.
\]
Equivalently, in terms of the intersection matrix $Q = (Q_{ij})$ we need a solution in positive real numbers $A_j$ of the system
\[
-\pi\sum_j Q_{ij}A_j/m_j = B_i,
\]
for arbitrary positive $B_i$. The existence of such a solution follows from condition (iii) on $\Gamma$ above. Moreover, we can suppose that the value of $\delta$ used in each application of Lemma \ref{basiclemma} is the same: we arrange that $2m_v\pi\delta < B_v $ for each $v$. Note that condition (ii) ensures that the lemma provides sufficiently many special points on each surface $C_v$ at which to perform the plumbing. 

Again we defer the proof of the final claim of the theorem. 
\end{proof}

\begin{remark} Note that $\Gamma$ need not be a tree in the above theorem. We can also allow $\Gamma$ to contain loops (edges from a vertex to itself), corresponding to immersed curves, at the cost of more restrictive hypotheses on the weights $m_v$.
\end{remark}

\begin{example}
Let $X_\Gamma$ be the plumbing of two disk bundles of degree $-1$ over $S^2$. Then (ii) is satisfied, but (iii) fails (for all $\B$). Correspondingly, there is no symplectic structure on $X_\Gamma$ with convex boundary: here $\partial X_\Gamma\cong S^1\times S^2$, and $X_\Gamma$ is a (smooth) blowup of $S^2\times D^2$. It follows from a result of Eliashberg \cite{eliash1} that $X_\Gamma$ cannot have convex boundary.
\end{example}

To expand on the negativity requirement on $\Gamma$ given in condition (iii) before the theorem, we have the following lemma.

\begin{lemma}\label{deflemma} Suppose $\Gamma$ is a connected plumbing graph with vertices having weights $-m_v$ and valences $d_v$, and assume $m_v\geq d_v$ for each $v$. Then the following are equivalent:
\begin{enumerate}
\item For every positive vector $\B$ there is a positive vector $\A$ such that $-Q_\Gamma \A = \B$.
\item For some positive vector $\B$ there is a positive vector $\A$ such that $-Q_\Gamma \A = \B$.
\item For some vertex $v$, we have strict inequality $m_v > d_v$.
\item $Q_\Gamma$ is negative definite.
\end{enumerate}
\end{lemma}

\begin{proof} Clearly (1) implies (2). Suppose $\A = (a_1,\ldots, a_n)$ is a positive vector such that every entry of $-Q_\Gamma\A$ is positive, and suppose contrary to (3) that every vertex of $\Gamma$ satisfies $m_v = d_v$. Equivalently, if $Q_\Gamma = (q_{ij})$, this says $\sum_j q_{ij} = 0$ for each $i$. Then for each $i$, we have
\begin{eqnarray*}
0 &<& -(Q_\Gamma\A)_i = -\sum_j q_{ij}a_j\\
&=& -\left( a_i(-\sum_{j\neq i}q_{ij}) + \sum_{j\neq i} q_{ij}a_j\right)\\
&=& \sum_{j\neq i} q_{ij}(a_i - a_j).
\end{eqnarray*}
Since $q_{ij}\geq 0$ for each $j\neq i$, this means that for each $i$ there exists $j$ such that $a_j < a_i$, which is absurd.

To see (3) implies (4), consider an arbitrary nonzero vector $\ww  = (w_1,\ldots,w_n)^T$. Let us write $s_i = \sum_j q_{ij} \geq 0$ for the row sums. Then
\begin{eqnarray*}
\ww^T Q_\Gamma\ww &=& \sum_{i,j} q_{ij} w_iw_j\\
&=& \frac{1}{2}\sum_{i,j} q_{ij}(w_i^2 + w_j^2 - (w_i - w_j)^2)\\
&=& \sum_{i,j} q_{ij} w_i^2 - \sum_{i<j} q_{ij}(w_i - w_j)^2\\
&=& \sum_i s_i w_i^2 - \sum_{i<j}q_{ij}(w_i - w_j)^2
\end{eqnarray*}
Since $s_i\leq 0$ and $q_{ij}\geq 0$ for $i\neq j$, the last expression is the sum of two non-positive terms. If the second term vanishes then $w_i = w_j$ for all $i\neq j$ with $q_{ij} \neq 0$. Since $\Gamma$ is connected, for any $i\neq j$ we can find a sequence $i = i_1, i_2, \ldots, i_m = j$ such that $q_{i_ki_{k+1}} \neq 0$ for each $k$, and therefore $w_{i} = w_{i_2} = \cdots = w_{i_m} = w_j$.  Since $i$ and $j$ were arbitrary we infer that $w_i = w_j$ for each $i,j$, but in this case the first term above is strictly negative as soon as some $s_i < 0$.

Finally, that (4) implies (1) was proved by Gay and Stipsicz (\cite{gaystip}, Lemma 3.3).

\end{proof}

\section{Lefschetz Fibrations}

We now show that a plumbing as in the previous section admits the structure of a Lefschetz fibration that is suitably compatible with the symplectic structure obtained above. 

\begin{theorem}\label{thm2} Let $X_\Gamma$ be a plumbing where $\Gamma$ satisfies (i), (ii) and (iii) above, equipped with the symplectic structure given by Theorem \ref{convthm}. Then there is a smooth map $\pi: X_\Gamma\to D^2$ with the following properties:
\begin{itemize}
\item $\pi$ is a Lefschetz fibration with bounded fibers, and the generic fiber of $\pi$ is diffeomorphic to the connected sum $\#_{v\in \Gamma}(C_v \#^{m_v - d_v} D^2)$

\item All singularities of $\pi$ lie in the fiber $\pi^{-1}(0)$. All regular fibers are properly embedded symplectic submanifolds, and $\pi^{-1}(0)$ is an immersed symplectic submanifold with only ordinary double points.
\item The configuration $\C = \bigcup_{v\in\Gamma} C_v$ of surfaces lies in $\pi^{-1}(0)$.
\item The Liouville field $V$ is outwardly transverse to $\partial\pi^{-1}(D^2(\epsilon))$ for each $\epsilon>0$, in particular transverse to the boundary of any ``nicely fibered'' neighborhood of $\C$.
\item The contact structure on $\partial X_\Gamma$ induced by $V$ is supported by the natural open book determined by $\pi$.
\end{itemize}
\end{theorem}

\begin{proof}

We work in the local models developed in the proof of Lemma \ref{basiclemma}: fix a surface $C$ and disk bundle $D\to C$ of degree $-m$, and marked points $x_0,x_1,\ldots, x_m$. Choose a nonnegative increasing function $\mu: [0,1]\to [0,1]$ satisfying
\begin{itemize}
\item $\mu(r) = r$ for $r\leq 1/3$
\item $\mu(r) = 1$ for $r \geq 2/3$.
\end{itemize}
We define the Lefschetz fibration first over $Z_0 = D - D_{x_0}$. Consider a Darboux trivialization of $D$ over the neighborhood $W_i$ of marked point $x_i$, $i>0$. For the constant $\delta$ chosen in Lemma \ref{basiclemma}, define
\[
\pi_0(r_1,\theta_1,r_2,\theta_2) = (\mu(r_1/\sqrt{\delta})\,\mu(r_2/\sqrt{\delta}), \theta_1 + \theta_2),
\]
in polar coordinates on $D^2$. Thus on $D|_{W_i}$, $\pi_0$ is a perturbation of the standard Lefschetz singularity. On $D|_{W_0} = (\sqrt{A-m\delta}, \sqrt{A})\times S^1 \times D^2(\sqrt{\delta})$, set
\[
\pi_0(r_1,\theta_1, r_2,\theta_2) = (\mu(r_2/\sqrt{\delta}), m\theta_1 + \theta_2).
\]
To extend $\pi_0$ over the rest of $Z_0$, observe that we have a function $\theta: \bigcup_{i\geq 0} (W_i - x_i) \to S^1$ given by $m\theta_1$ on $W_0$ and by $\theta_1$ on $W_i$, $i>0$. Our orientation convention for $W_0$ shows that this function has degree 0 on the boundary $\partial (C - \bigcup_{i\geq 0}W_i)$, hence extends over the rest of $C$. Now for $p\in C - \bigcup W_i$ we define 
\[
\pi_0(p, r_2,\theta_2) = (\mu(r_2/\sqrt{\delta}), \theta(p) + \theta_2).
\]
This gives a well-defined map $\pi_0: D|_{C - x_0}\to D^2$, and it is straightforward to see that every $y \in D^2(1/9)- 0$ is a regular value.

For $p \in Z_1 = D^2(\sqrt{m\delta}) \times D^2(\sqrt{\delta})$ we set $\pi_1(p) = (\mu(r_2/\sqrt{\delta}), \theta_2)$. Then it is easy to see that $\pi_1 = \pi_0\circ \phi$ where $\phi$ is the gluing map $T\subset Z_1\to Z_0$ from previously, hence we obtain a smooth function $\pi: D\to D^2$. Now replace $D$ by the smaller subset $\pi^{-1}(D^2(1/9))$. Since $C$ obviously lies in $\pi^{-1}(0)$, this replacement is diffeomorphic to $D$. Moreover, $\pi$ clearly restricts as a Lefschetz fibration having $m$ singularities; the unique singular fiber is $\pi^{-1}(0)$. The latter is clearly given by $C \cup D_{x_1}\cup\cdots\cup D_{x_m}$. It is also easy to see that a nonsingular fiber is symplectic, from the local forms over each $W_i$.

Now we form the plumbing as in Theorem \ref{convthm}: since $\pi_{W_i}$ is invariant under the involution used in the plumbing construction, we see that the Lefschetz fibrations constructed on each disk bundle glue together to give a Lefschetz fibration $\pi: X_{\Gamma}\to D^2(1/9)$. (As before, we arrange to use the same value of $\delta$ in Lemma \ref{basiclemma} for each surface $C_v$ in the plumbing.)

To understand the generic fiber of $\pi$, recall that one can obtain the generic fiber of a Lefschetz fibration from any singular fiber by smoothing the ordinary double points appearing in the latter. In our case the singular fiber in one of the disk bundles $D\to C_v$ is the union of $C_v$ with $m_v$ normal disks; of the latter $d_v$ are identified in the plumbing with disks lying in other surfaces $C_{v'}$. This easily gives the description of the general fiber of $\pi$ in the statement of the theorem. 

Finally, to see that the contact structure on $\partial X_\Gamma$ induced by $V$ is supported by the open book induced by $\pi$, recall that for a contact structure $\xi =\ker\alpha$ to be supported by an open book means two things: the differential $d\alpha$ restricts as a positive area form to the pages, and the binding is positively transverse to $\xi$. The first of these follows in our case since the fibers of $\pi$ are symplectic, and the second from a quick check in the local models near singular points.
\end{proof}

Note that Theorem \ref{thm2} provides a proof of the missing parts of Lemma \ref{basiclemma} and Theorem \ref{convthm} regarding convexity, though strictly one must shrink the fibered neighborhoods $\pi^{-1}(\epsilon)$ near the special points of $C_v$ that were not used in the plumbing construction.

\end{document}